\documentclass{ws-jktr}
\usepackage{hyperref}
\usepackage{float}
\floatstyle{plaintop}
\restylefloat{table}

\newcommand{\C}{\mathbb{C}}

\begin{document}

\markboth{Ji-Young Ham, Joongul Lee}
{Explicit formula for A-polynomial of $C(2n, 3)$}

\catchline{}{}{}{}{}

\title{An explicit formula for the $A$-polynomial of the knot with Conway's notation $C(2n, 3)$
}

\author{Ji-Young Ham, Joongul Lee
}

\address{Department of Science, Hongik University, 
94 Wausan-ro, Mapo-gu, Seoul,
04066,\\
 Department of Mathematical Sciences, Seoul National University, 
  1 Gwanak-ro, Gwanak-gu, Seoul 08826  \\
   Korea\\
   jiyoungham1@gmail.com}

\author{Joongul Lee
}

\address{Department of Mathematics Education, Hongik University, 
94 Wausan-ro, Mapo-gu, Seoul,
 04066\\
   Korea\\
   jglee@hongik.ac.kr}

\maketitle

\begin{abstract}
An explicit formula for the $A$-polynomial of the knot with Conway's notation $C(2n,3)$ is obtained from the explicit Riley-Mednykh polynomial of it. 
\end{abstract}



\section{Introduction}
In 1994, the $A$-polynomial, $A(L,M)$, of a compact 3-manifold $N$ with a single torus boundary was introduced by Cooper, Culler, Gillet, Long, and Shalen 
in~\cite{CCGLS1}. It's variables are eigenvalues of the meridian and the longitude under the representations from $\pi_1{N}$ into 
$\text{\textnormal{SL}}(2,C)$. One of the main results of~\cite{CCGLS1} is ``boundary slopes are boundary slopes".  That is the boundary slope of the Newton polygon of $A(L,M)$ is the boundary slope of an incompressible surface in $N$. In 2001, it was shown that the Newton polygon of $A(L,M)$ is dual to the fundamental polygon of the Culler-Shalen seminorm~\cite{BZ1}. The Culler-Shalen seminorm~\cite{CGLS} can be used to detect and classify the exceptional surgeries, which is a step toward another proof of the Poincare conjecture. $A$-polynomial also encodes the deformed structure of $N$. For example, using the longitude $L$, in~\cite{HMP,HL1},  the volumes of the deformed cone-manifolds are computed and in~\cite{HL,HL2}, the Chern-Simons invariants~\cite{CS,HLM3} of the deformed orbifolds are computed. 
The non-commutative $A$-polynomial $A(L,M,q)$ of a knot is introduced and it is conjectured that $A(L,M,1)=B(M) A(L,M^{1/2})$ for some polynomial $B(M)$ of $M$ and the conjecture is called $AJ$ conjecture~\cite{Garou}. $AJ$ conjecture is proved for some knots. For example, our knot, the knot with Conway's notation $C(2n,3)$ satisfies the $AJ$ conjecture~\cite{LT}. If $AJ$ conjecture is true then the colored Jones polynomial detects knottedness~\cite{DG} as $A$-polynomial. 

With today's technology, $A$-polynomial is relatively difficult to compute. Recovering representations from a triangulation of $N$ and compute a factor of the 
$A$-polynomial is another try to compute it~\cite{Z2}. By one by one computation, $A$-polynomials are known up to $8$ crossings and most $9$ crossings and many $10$ crossings. For infinite families, recursively, $A$-polynomials are known for twist knots~\cite{HS}, $(-2,3,1+2n)$ pretzel knots~\cite{TY,GT}, $J[m,2n]$~\cite{HS} for $m$ between $2$ and $5$~\cite{Pe} and explicitly, for two-bridge torus knots~\cite{CCGLS1,HS}, iterated torus knots~\cite{Ni}, and for twist knots~\cite{Mat,HL}. We record here that $J[3,-2n]$ is the mirror image of $C(2n,3)$.
\medskip

The main purpose of the paper is to find the explicit formula for the $A$-polynomial of the knot with Conway's notation $C(2n,3)$. Let us denote the knot with Conway's notation $C(2n,3)$ by $T_{2n}$ and the $A$-polynomial of the knot with Conway's notation $C(2n,3)$ by $A_{2n}$. The following theorem gives the explicit formula for the $A$-polynomial of $T_{2n}$.

\begin{theorem}  \label{thm:A-polynomial}
A-polynomial $A_{2n}=A_{2n}(L,M)$ is given explicitly by 

\medskip
\begin{equation*}
A_{2n} = \begin{cases}
 \sum_{i=0}^{2n} 
\left(\begin{smallmatrix}
n+ \lfloor \frac{i}{2} \rfloor \\
i
\end{smallmatrix}\right)
 \left( \frac{(L M^{4 n}-1) (1-M^{2})}{1+L M^{2+4 n}} \right)^i 
 \left( \frac{1+ L M^{6+4 n}}{M^2+L M^{4+4 n}} \right)^{\lfloor \frac{1+i}{2}  \rfloor} \\
\times M^{-2 n} \left(1+L M^{2+4 n}\right)^{3 n} 
\qquad
\text{if $n \geq 0$} \\
 \sum_{i=0}^{-2n-1} 
 \left(\begin{smallmatrix}
-n+ \lfloor \frac{(i-1)}{2} \rfloor \\
i
\end{smallmatrix}\right)
 \left(\frac{(1-M^2) (M^{-4 n}-L)}{L M^2+M^{-4 n}}\right)^i  
 \left( \frac{L M^6+M^{-4 n}}{L M^4+M^{2-4 n}}\right)^{\lfloor \frac{1+i}{2}  \rfloor} \\
\times M^{8 n+6} \left(L M^2+M^{-4 n}\right)^{-3 n-1}  \qquad
\text{if $n<0$}
\end{cases}
\end{equation*}
\end{theorem}

\medskip 

One can consult~\cite{HL} for solving the recurrence formula. Our writing is parallel with that in~\cite{Mat} which is based on~\cite{HS}.
\section{Proof of Theorem~\ref{thm:A-polynomial}}

\begin{figure} 
\begin{center}
\resizebox{3.5cm}{!}{\includegraphics{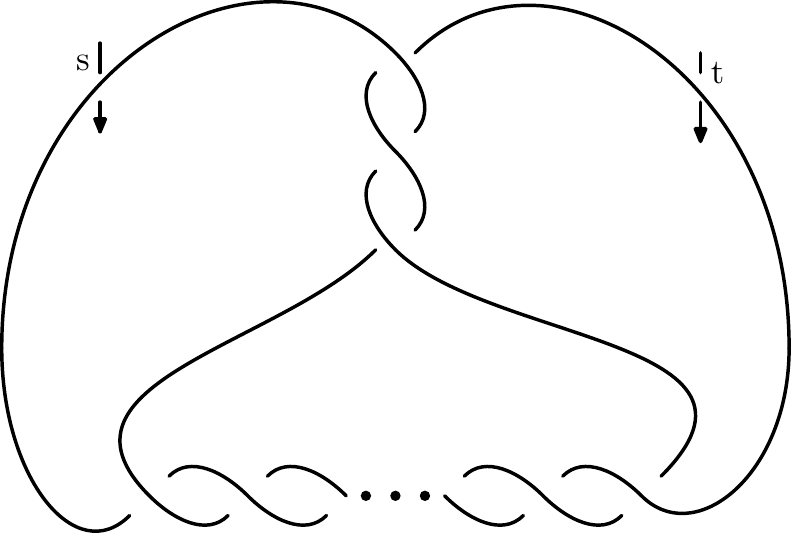}}
\qquad
\resizebox{3.5cm}{!}{\includegraphics{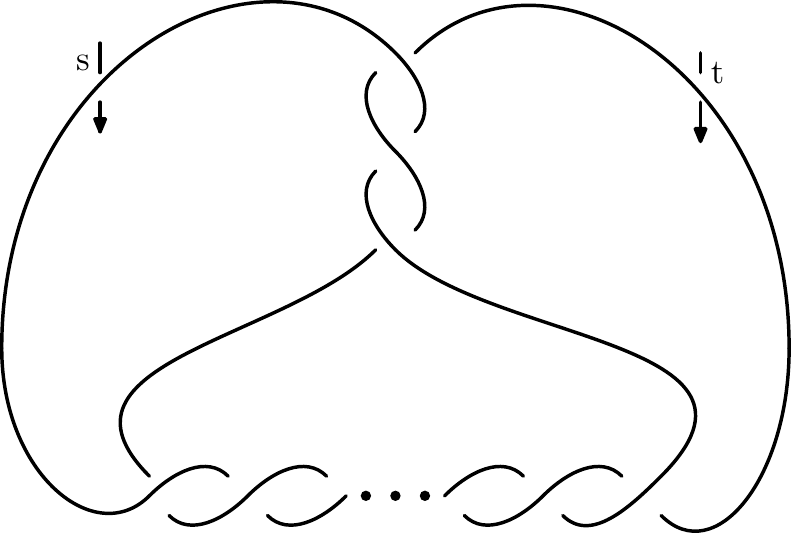}}
\caption{A two bridge knot with Conway's notation $C(2n,3)$  for $n>0$ (left) and for $n<0$ (right)} \label{fig:C[2n,3]}
\end{center}
\end{figure}

A knot $K$ is a two bridge knot with Conway's notation $C(2n,3)$ if $K$ has a regular two-dimensional projection of the form in Figure~\ref{fig:C[2n,3]}. Recall that we denote it by $T_{2 n}$. Let us denote the exterior of $T_{2n}$ by $X_{2 n}$. The following proposition gives the fundamental group of 
$X_{2n}$~\cite{HL1,HL2,HS,R1}.
\begin{proposition}\label{theorem:fundamentalGroup}
$$\pi_1(X_{2n})=\left\langle s,t \ |\ swt^{-1}w^{-1}=1\right\rangle,$$
where $w=(ts^{-1}tst^{-1}s)^n$.
\end{proposition}
\medskip 

Given  a set of generators, $\{ s,t \}$, of the fundamental group for 
$\pi_1 (X_{2n})$, we define a representation $\rho \ : \ \pi_1 (X_{2n}) \rightarrow \text{SL}(2, \C)$ by

\begin{center}
$$\begin{array}{ccccc}
\rho(s)=\left[\begin{array}{cc}
                       M &       1 \\
                        0      & M^{-1}  
                     \end{array} \right]                          
\text{,} \ \ \
\rho(t)=\left[\begin{array}{cc}
                   M &  0      \\
                   2-M^2-M^{-2}-x      & M^{-1} 
                 \end{array}  \right].
\end{array}$$
\end{center}

Then $\rho$ can be identified with the point $(M,x) \in \C^2$.  When $M$ varies we have an algebraic set 
whose defining equation is the following explicit Riley-Mednykh polynomial.

\begin{lemma} \label{lem1:RMpolynomial}
$\rho$ is a representation of $\pi_1(X_{2n})$ if and only if $x$ is a root of the following Riley-Mednykh polynomial $P_{2n}=P_{2n}(x,M)$ which is given explicitly by 

\begin{equation*}
P_{2n} = \begin{cases}
 \sum_{i=0}^{2n} 
\left(\begin{smallmatrix}
n+ \lfloor \frac{i}{2} \rfloor \\
i
\end{smallmatrix}\right)
M^{4 n} \left(M^2+M^{-2}+x-1\right)^{i} 
\left(-x\right)^{\lfloor \frac{1+i}{2}  \rfloor}
\qquad
\text{if $n \geq 0$}, \\
 \sum_{i=0}^{-2n-1} 
\left(\begin{smallmatrix}
-n+ \lfloor \frac{(i-1)}{2} \rfloor \\
i
\end{smallmatrix}\right)
 M^{-4 n-2} \left(-M^2-M^{-2}-x+1 \right)^i 
 \left(-x \right)^{\lfloor \frac{1+i}{2}  \rfloor}
\text{if $n<0$}.
\end{cases}
\end{equation*}
\end{lemma}

\begin{proof}
In~\cite{HL1}, $P_{2n}$ is give by the following recursive formula.
\begin{equation*}
P_{2n} = \begin{cases}
 Q P_{2(n-1)} -M^8 P_{2(n-2)} \ \text{if $n>1$}, \\
 Q P_{2(n+1)}-M^8 P_{2(n+2)} \ \text{if $n<-1$},
\end{cases}
\end{equation*}
\medskip

\noindent with initial conditions
\begin{align*}
& P_{-2}   =M^2 x^2+\left(M^4- M^2+1\right) x+M^2,\\
& P_{0}  =M^{-2} \ \text{\textnormal{for}} \ n < 0 \qquad \text{\textnormal
{and}} \qquad P_{0}  =1 \ \text{\textnormal{for}} \ n > 0,\\    
& P_{2}  =-M^4 x^3+\left(-2M^6+M^4-2 M^2\right) x^2+\left(-M^8+M^6-2 M^4+ M^2-1\right) x+M^4, \\
\end{align*}
\noindent and
$$Q=-M^4 x^3+\left(-2M^6+2 M^4-2M^2\right) x^2+\left(-M^8+2M^6-3M^4+2M^2-1\right) x+2M^4.$$

\medskip

We write $f_{2n}$ for the claimed formula and show that $f_{2n}=P_{2n}$. 

\noindent Case I: $n \geq 0$. When $i > 2n$ or $i<0$, 
$\left(\begin{smallmatrix}
n+ \lfloor \frac{i}{2} \rfloor \\
i
\end{smallmatrix}\right)$
is undefined and can be considered as zero. Hence the finite sum can be regarded as an infinite sum.
Direct computation shows that $f_0=P_0$ and $f_2=P_2$.
Now, we only need to show that $f_{2n}$ satisfies the recursive relation.
Note that $Q$ can be written as $M^4 \left(-x \left( M^2+M^{-2}+x-1\right)^2+2 \right)$.
\begin{align*}
&Q f_{2(n-1)} -M^8 f_{2(n-2)} \\
&=
M^4 \left( -x \left( M^2+M^{-2}+x-1\right)^2+2\right) \\
& \times  \sum_i
\left(\begin{smallmatrix}
n-1+ \lfloor \frac{i}{2} \rfloor \\
i
\end{smallmatrix}\right)
M^{4 n-4} \left(M^2+M^{-2}+x-1\right)^{i} 
\left(-x\right)^{\lfloor \frac{1+i}{2}  \rfloor}\\
&-M^8
\sum_i
\left(\begin{smallmatrix}
n-2+ \lfloor \frac{i}{2} \rfloor \\
i
\end{smallmatrix}\right)
M^{4 n-8} \left(M^2+M^{-2}+x-1\right)^{i} 
\left(-x\right)^{\lfloor \frac{1+i}{2}  \rfloor} \\
& = \sum_i \left[
\left(\begin{smallmatrix}
n-2+ \lfloor \frac{i}{2} \rfloor \\
i-2
\end{smallmatrix}\right)
+2 
\left(\begin{smallmatrix}
n-1+ \lfloor \frac{i}{2} \rfloor \\
i
\end{smallmatrix}\right)
-\left(\begin{smallmatrix}
n-2+ \lfloor \frac{i}{2} \rfloor \\
i
\end{smallmatrix}\right) \right] \\
& \times M^{4 n} \left(M^2+M^{-2}+x-1\right)^{i} \left(-x\right)^{\lfloor \frac{1+i}{2}  \rfloor} \\
&= \sum_i
\left(\begin{smallmatrix}
n+ \lfloor \frac{i}{2} \rfloor \\
i
\end{smallmatrix}\right)
M^{4 n} \left(M^2+M^{-2}+x-1\right)^{i} 
\left(-x\right)^{\lfloor \frac{1+i}{2}  \rfloor} \\
&=f_{2n}
\end{align*}
In the last equality we use the binomial relation 
$$\left(\begin{smallmatrix}
a \\
b
\end{smallmatrix}\right)
=
\left(\begin{smallmatrix}
a-1\\
b-1
\end{smallmatrix}\right)
+
\left(\begin{smallmatrix}
a-1 \\
b
\end{smallmatrix}\right)$$
three times.

\noindent Case II: $n < 0$. When $i > -2n-1$ or $i<0$, 
$\left(\begin{smallmatrix}
-n+ \lfloor \frac{i-1}{2} \rfloor \\
i
\end{smallmatrix}\right)$
is undefined and can be considered as zero. Hence the finite sum can be regarded as an infinite sum.
Direct computation shows that $f_0=P_0$ and $f_{-2}=P_{-2}$.
Now, we only need to show that $f_{2n}$ satisfies the recursive relation.
We know that $Q$ can be written as $M^4 \left(-x \left( M^2+M^{-2}+x-1\right)^2+2 \right)$.
\begin{align*}
&Q f_{2(n+1)} -M^8 f_{2(n+2)} \\
&=
M^4 \left( -x \left( M^2+M^{-2}+x-1\right)^2+2\right) \\
& \times  \sum_i
\left(\begin{smallmatrix}
-n-1+ \lfloor \frac{(i-1)}{2} \rfloor \\
i
\end{smallmatrix}\right)
M^{-4 n-6} \left(M^2+M^{-2}+x-1\right)^{i} 
\left(-x\right)^{\lfloor \frac{1+i}{2}  \rfloor}\\
&-M^8
\sum_i
\left(\begin{smallmatrix}
-n-2+ \lfloor \frac{i-1}{2} \rfloor \\
i
\end{smallmatrix}\right)
M^{-4 n-10} \left(M^2+M^{-2}+x-1\right)^{i} 
\left(-x\right)^{\lfloor \frac{1+i}{2}  \rfloor} \\
& = \sum_i \left[
\left(\begin{smallmatrix}
-n-2+ \lfloor \frac{i-1}{2} \rfloor \\
i-2
\end{smallmatrix}\right)
+2 
\left(\begin{smallmatrix}
-n-1+ \lfloor \frac{i-1}{2} \rfloor \\
i
\end{smallmatrix}\right)
-\left(\begin{smallmatrix}
-n-2+ \lfloor \frac{i-1}{2} \rfloor \\
i
\end{smallmatrix}\right) \right] \\
&\times M^{-4 n-2} \left(M^2+M^{-2}+x-1\right)^{i} \left(-x\right)^{\lfloor \frac{1+i}{2}  \rfloor} \\
&= \sum_i
\left(\begin{smallmatrix}
-n+ \lfloor \frac{i-1}{2} \rfloor \\
i
\end{smallmatrix}\right)
M^{-4 n-2} \left(M^2+M^{-2}+x-1\right)^{i} 
\left(-x\right)^{\lfloor \frac{1+i}{2}  \rfloor} \\
&=f_{2n}
\end{align*}
In the last equality we use the binomial relation 
$$\left(\begin{smallmatrix}
a \\
b
\end{smallmatrix}\right)
=
\left(\begin{smallmatrix}
a-1\\
b-1
\end{smallmatrix}\right)
+
\left(\begin{smallmatrix}
a-1 \\
b
\end{smallmatrix}\right)$$
three times, again.
\end{proof}

   Let $l = ww^{*}s^{-4n}$~\cite{CCGLS1,HS}, where $w^{*}$ is the word obtained by reversing $w$. Let $L=\rho(l)_{11}$. Then $l$ is the longitude which is null-homologus in $X_{2n}$ (you can read a twisted longitude $ww^{*}$ from the Schubert normal form of the knot $C(2n,3)$ and multiply it by $s^{-4n}$ so that the exponent sum of $l$ becomes $0$). And we have
   
\medskip 
   
\begin{lemma}~\cite{HL1,HL2}
\label{lem2:longitude}
\begin{align*}
L &=-M^{-4n-2}\frac{M^{-2}+x}{ M^{2}+x}, \\
x &=-\frac{1+L M^{6 + 4 n}}{M^2 (1 + L M^{2 + 4 n})}.
\end{align*}
\end{lemma}
Now substituting $-\frac{1+L M^{6 + 4 n}}{M^2 (1 + L M^{2 + 4 n})}$ for $x$ into $P_{2n}$, for $n \geq 0$, gives

\begin{align*}
\sum_{i=0}^{2n} 
\left(\begin{smallmatrix}
n+ \lfloor \frac{i}{2} \rfloor \\
i
\end{smallmatrix}\right)
&M^{4 n} \left(M^2+M^{-2}-\frac{1+L M^{6 + 4 n}}{M^2 (1 + L M^{2 + 4 n})}-1\right)^{i} \\
& \times \left(\frac{1+L M^{6 + 4 n}}{M^2 (1 + L M^{2 + 4 n})}\right)^{\lfloor \frac{1+i}{2}  \rfloor}.
\end{align*}

We observe that 
\begin{align*}
M^2+M^{-2}-\frac{1+L M^{6 + 4 n}}{M^2 (1 + L M^{2 + 4 n})}-1=\frac{(L M^{4 n}-1) (1-M^{2})}{(1 + L M^{2 + 4 n})}.
\end{align*}

The resulting expression,

\begin{align*}
 \sum_{i=0}^{2n} 
\left(\begin{smallmatrix}
n+ \lfloor \frac{i}{2} \rfloor \\
i
\end{smallmatrix}\right)
M^{4 n} \left( \frac{(L M^{4 n}-1) (1-M^{2})}{1+L M^{2+4 n}} \right)^i 
 \left( \frac{1+ L M^{6+4 n}}{M^2+L M^{4+4 n}} \right)^{\lfloor \frac{1+i}{2}  \rfloor},
\end{align*}
once denominators are cleared and some power of $M$ is factored out to give a polynomial, gives the A-polynomial $A_{2n}(L,M)$.
We multiply it by 
$M^{-6n} (1 + L M^{2 + 4 n})^{3 n}$ so that we have the claimed formula in Theorem~\ref{thm:A-polynomial}.
The following equality, which guarantees that the claimed formula has the constant term $1$ and is a polynomial, 
\begin{align*}
c_{2n}=\sum_{i=0}^{2n} 
\left(\begin{smallmatrix}
n+ \lfloor \frac{i}{2} \rfloor \\
i
\end{smallmatrix}\right)
 \left(M^{2}-1\right)^i \left( \frac{1}{M^2} \right)^{\lfloor \frac{1+i}{2}  \rfloor}  M^{-2 n} =1,
 \end{align*}
 can be proved by induction,
 \begin{align*}
 c_{2n}= \frac{(M^2-1)^2 c_{2(n-1)}}{M^4}+\frac{2 c_{2(n-1)}}{M^2}-\frac{c_{2(n-2)}}{M^4},
 \end{align*}
 which is proved in the following.
 As in the case of the proof of Lemma~\ref{lem1:RMpolynomial}, $c_{2n}$ can be regarded as an infinite sum. Direct computation shows that  $f_{2}=c_{2}$ and $f_{4}=c_{4}$.
 Let $f_{2n}$ be the right side of the claimed formula. We will show that $f_{2n}=c_{2n}$.
 Again as in the case of the proof of Lemma~\ref{lem1:RMpolynomial}, $f_{2n}$ can be written as 
 \begin{align*}
&\sum_i \left[\binom{n-2+\lfloor \frac{i}{2}  \rfloor}{i-2}+2 \binom{n-1+\lfloor \frac{i}{2}  \rfloor}{i}-\binom{n-2+\lfloor \frac{i}{2}  \rfloor}{i}\right]\\
& \times \left(M^{2}-1\right)^i \left( \frac{1}{M^2} \right)^{\lfloor \frac{1+i}{2}  \rfloor}  M^{-2 n}.
 \end{align*}
 Hence as in the case of the proof of Lemma~\ref{lem1:RMpolynomial}, by using the binomial relations three times, we have $f_{2n}=c_{2n}$.

Similarly, for $n < 0$, substituting 
$$-\frac{1+L M^{6 + 4 n}}{M^2 (1 + L M^{2 + 4 n})}=-\frac{M^{-4n}+L M^{6}}{M^2 (M^{-4n} + L M^{2})}$$
 for $x$ into $P_{2n}$ gives

\begin{align*}
 \sum_{i=0}^{-2n-1} 
\left(\begin{smallmatrix}
-n+ \lfloor \frac{(i-1)}{2} \rfloor \\
i
\end{smallmatrix}\right)
& M^{-4 n-2} \left(-M^2-M^{-2}+\frac{M^{-4n}+L M^{6}}{M^2 (M^{-4n} + L M^{2})}+1 \right)^i \\
& \times \left(\frac{M^{-4n}+L M^{6}}{M^2 (M^{-4n} + L M^{2})}\right)^{\lfloor \frac{1+i}{2}  \rfloor}.
\end{align*}

We observe that 
\begin{align*}
-M^2-M^{-2}+\frac{M^{-4n}+L M^{6}}{M^2 (M^{-4n} + L M^{2})}+1=\frac{(M^{-4 n}-L) (1-M^{2})}{(M^{-4 n} + L M^{2})}.
\end{align*}

The resulting expression,

\begin{align*}
 \sum_{i=0}^{-2n-1} 
 \left(\begin{smallmatrix}
-n+ \lfloor \frac{(i-1)}{2} \rfloor \\
i
\end{smallmatrix}\right)
 M^{-4 n-2}  \left(\frac{(1-M^2) (M^{-4 n}-L)}{L M^2+M^{-4 n}}\right)^i  
 \left( \frac{L M^6+M^{-4 n}}{L M^4+M^{2-4 n}}\right)^{\lfloor \frac{1+i}{2}  \rfloor},
\end{align*}
once denominators are cleared and some power of $M$ is factored out to give a polynomial, gives the A-polynomial $A_{2n}(L,M)$.
We multiply it by $M^{12 n+8} \left(L M^2+M^{-4 n}\right)^{-3 n-1}$ so that we have the claimed formula in Theorem~\ref{thm:A-polynomial}:

\begin{align*}
&\sum_{i=0}^{-2n-1} 
 \left(\begin{smallmatrix}
-n+ \lfloor \frac{(i-1)}{2} \rfloor \\
i
\end{smallmatrix}\right)
 \left(\frac{(1-M^2) (M^{-4 n}-L)}{M^2 (L+M^{-4 n-2})}\right)^i  
 \left( \frac{M^2 (L +M^{-4 n-6})}{L+M^{-4 n-2}}\right)^{\lfloor \frac{1+i}{2}  \rfloor} \\
&\times M^{2 n+4} \left(L +M^{-4 n-2}\right)^{-3 n-1}.
\end{align*}

\noindent Now we want to show that the claimed formula does not have fractions. For $n=-1$, by direct computation, one can show that the claimed formula does not have fractions. For each $n <-1$, fractions can only occur in the following sums.
\begin{align*}
& \sum_{i=0}^{-2n-1} 
 \left(\begin{smallmatrix}
-n+ \lfloor \frac{(i-1)}{2} \rfloor \\
i
\end{smallmatrix}\right)
 \left((1-M^2) (-L)\right)^i  
 L^{\lfloor \frac{1+i}{2}  \rfloor} \\
& \times M^{2 n+4-2 i+2 \lfloor \frac{1+i}{2}  \rfloor} L^{-3 n-1-i-\lfloor \frac{1+i}{2}  \rfloor} \\
&=\sum_{i=0}^{-2n-1} 
 \left(\begin{smallmatrix}
-n+ \lfloor \frac{(i-1)}{2} \rfloor \\
i
\end{smallmatrix}\right)
(M^2-1) ^i  
 M^{2 n+4-2 i+2 \lfloor \frac{1+i}{2}\rfloor}   L^{-3 n-1} .
\end{align*}
\noindent Let $c_{2n}$ be the coefficient of $L^{-3 n-1}$ of the above sum. Then,
 one can prove that $c_{2n}=M^{4}$
by the following recurrence relation,
 \begin{align*}
 c_{2n}= \frac{(M^2-1)^2 c_{2(n+1)}}{M^4}+\frac{2 c_{2(n+1)}}{M^2}-\frac{c_{2(n+2)}}{M^4},
 \end{align*}
 which can be proved as in the case of $n>0$.
 
\noindent Now, we are going to compute a part of the coefficient of $L^{-3n-2}$.
 For each $n<0$, the term $-L^{-3n-2}$ exists:
 \begin{align*}
 & \sum_{i=-2n-2}^{-2n-1} 
 \binom{-n+ \lfloor \frac{(i-1)}{2} \rfloor}{i} 
 M^{2 n+4-2 i+2 \lfloor \frac{1+i}{2}\rfloor}\\
 & \times \left(-L\right)^i  \left( \lfloor \frac{1+i}{2}  \rfloor L^{\lfloor \frac{1+i}{2}  \rfloor-1}M^{-4n-6}\right)
  L^{-3 n-1-i-\lfloor \frac{1+i}{2}  \rfloor}\\
  &= \sum_{i=-2n-2}^{-2n-1}  \lfloor \frac{1+i}{2}  \rfloor L^{-3n-2}.
 \end{align*}
\noindent And now we are going to compute a part of the coefficient of $L^{0}$. For each $n<0$, the term
$M^{12 n^2+14 n+6}$ exists:
 \begin{align*}
 & \sum_{i=-2n-1}^{-2n-1} 
 \binom{-n+ \lfloor \frac{(i-1)}{2} \rfloor}{i} 
 M^{2 n+4-2 i+2 \lfloor \frac{1+i}{2}\rfloor}\\
 & \times \left(M^{-4 n}\right)^i \left(M^{-4 n-6}\right)^{\lfloor \frac{1+i}{2}\rfloor}  (M^{-4 n-2})^{-3 n-1-i-\lfloor \frac{1+i}{2}  \rfloor}\\
 &=\sum_{i=-2n-1}^{-2n-1} M^{12 n^2+12 n+6-2 \lfloor \frac{1+i}{2}  \rfloor}.
 \end{align*}
\noindent Hence there does not exist redundant $L$ or $M$ factors.
\medskip

\section*{Acknowledgments}
The authors would like to thank Hyuk Kim, Yi Ni, Daniel Mathews, and anonymous referees.
The first author was supported by Basic Science Research Program through the National Research Foundation of Korea (NRF) funded by the Ministry of Education, Science and Technology (No. NRF-2008-341-C00004 and NRF-R01-2008-000-10052-0). The second author was supported by 2016 Hongik University Research Fund.

\end{document}